\documentclass[journal]{IEEEtran}
\usepackage{stmaryrd}
\usepackage{amsfonts}
\usepackage{graphicx,times,amsmath}
\usepackage{graphics,color}
\usepackage{graphicx}
\usepackage{latexsym}
\usepackage{amsmath}
\usepackage{amsfonts}
\usepackage{amssymb}
\usepackage{url}
\usepackage{subfigure}
\usepackage{times}
\usepackage{color}
\usepackage{float}
\usepackage{graphics,color}
\usepackage{graphicx}
\usepackage{amssymb,amsmath,amsfonts}
\usepackage{amsmath,mathrsfs,bm,url,times}
\usepackage{latexsym}
\usepackage{subfigure}
\usepackage{algorithmic}
\usepackage{algorithm}
\usepackage{epsfig,url}

\newtheorem{theorem}{Theorem}
\newtheorem{lemma}{Lemma}

\newtheorem{definition}{Definition}

\newtheorem{remark}{Remark}
\newtheorem{corollary}{Corollary}
\newtheorem{property}{Property}
\newenvironment{proof}{\vspace{1ex}\noindent{\bf Proof.}\hspace{0.5em}}
    {\hfill\qed\vspace{1ex}}
\def\qed{\hfill \vrule height 6pt width 6pt depth 0pt}

\newcommand{\onetom}{1,\cdots,m}

\newcommand{\onetoK}{1,\cdots,K}

\allowdisplaybreaks

\title{\ \\ \LARGE\bf
Pull-Based Distributed Event-triggered Consensus for Multi-agent Systems
with Directed Topologies
\thanks{Xinlei Yi is with the School of Mathematical Sciences, Fudan University;
Wenlian Lu (corresponding author) is with the Centre for Computational
Systems Biology and School of Mathematical Sciences, Fudan University, and
Department of Computer Science, The University of Warwick, Coventry, United
Kingdom; Tianping Chen is with the School of Computer Science and School of
Mathematical Sciences, Fudan University, Shanghai 200433, China (email:
\{yix11, wenlian, tchen\}@fudan.edu.cn).} \thanks{This work is jointly
supported by the Marie Curie International Incoming Fellowship from the
European Commission (FP7-PEOPLE-2011-IIF-302421), the National Natural
Sciences Foundation of China (Nos. 61273211 and 61273309), and the Program
for New Century Excellent Talents in University (NCET-13-0139).}}

\author{Xinlei Yi, Wenlian Lu and Tianping Chen}

\begin{document}
\maketitle
\begin{abstract}

This paper mainly investigates consensus problem with pull-based
event-triggered feedback control. For each agent, the diffusion coupling
feedbacks are based on the states of its in-neighbors at its latest
triggering time and the next triggering time of this agent is determined by
its in-neighbors' information as well. The general directed topologies,
including irreducible and reducible cases, are investigated. The scenario
of distributed continuous communication is considered firstly. It is proved that if
the network topology has a spanning tree, then the event-triggered coupling
strategy can realize consensus for the multi-agent system. Then the results
are extended to discontinuous communication, i.e., self-triggered control,
where each agent computes its next triggering time in advance without
having to observe the system's states continuously. The effectiveness of
the theoretical results are
illustrated by a numerical example finally.\\

\end{abstract}
\section{Introduction}
Consensus problem in multi-agent systems has been
widely and deeply investigated. The basic idea of consensus lies in that
each agent updates its state based on its own state and the states of its
neighbors in such a way that the final states of all agents converge to a
common value \cite{Fcz}. The model normally is of the following form:
\begin{align}\label{mg4}
\dot{x}(t)=-Lx(t)
\end{align}
where the column vector $x(t)$ consists of all nodes' states and $L$ is the
corresponding weighted Laplacian matrix. There are many results reported in
this field  \cite{Fcz}-\cite{Liubo} and the references therein. In these
researches, the network topologies vary from fixed topologies to
stochastically switching topologies, and the most basic condition to
realize a consensus is that the underlying graph of the network system has
a spanning tree.

In recent years, with the development of sensing, communications, and
computing equipment, event-triggered control \cite{Pta}-\cite{Crf} and
self-triggered control \cite{Aapt}-\cite{Aap} have been proposed and
studied. Instead of using the continuous state to realize a consensus, the
control in event-triggered control strategy is piecewise constant between
the triggering times which need to be determined. 
Self-triggered control is a natural extension of the event-triggered
control since the derivative of the concern multi-agent system's state is
piecewise constant, which is very easy to work out solutions (agents'
states) of the system. In particular, each agent predicts its next
triggering time at the previous one. Inspired by above idea of event-triggered control and self-triggered control, \cite{Dvd}-\cite{Now} considered the consensus problem for multi-agent systems with event-triggered control. In particular,  in \cite{Dvd}, under the condition
that the graph is undirected and strongly connected, the authors provide
event-triggered and self-triggered approaches in both centralized and
distributed formulations. It should be emphasized that the approaches
cannot be applied to directed graph. In \cite{Zliu}, the authors
investigate the average-consensus problem of multi-agent systems with
directed and weighted topologies, but they need an additional assumption
that the directed topology must be balanced. In \cite{Yfg}, the authors
propose a new combinational measurement approach to event design, which
will be used in this paper.

In this paper, continuing with previous works, we study event-triggered and
self-triggered consensus in multi-agent system with directed, reducible
(irreducible) and weighted topology.

Consider the following continuous-time linear multi-agent system with
discontinuous diffusions as follows
\begin{align}
\begin{cases}
\dot{x}_{i}(t)=u_{i}(t)\\
u_{i}(t)=-\sum_{j=1}^{m}L_{ij}x_{j}(t_{k_{i}(t)}^{i}),~i=\onetom\end{cases}\label{mg3}
\end{align}
where $k_{i}(t)=arg\max_{k}\{t^{i}_{k}\le t\}$, the increasing time sequence $\{t_{k}^{j}\}_{k=1}^{\infty}$, $j=\onetom$, which is named as {\em trigger times}, is agent-wise and normally assuming $t_{1}^{j}=0$, for all $j\in \mathcal I$, where $\mathcal I=\{1, 2,\cdots,
m\}$. We say agent $v_i$ triggers at time $t^{i}_{k}$ means agent $v_{i}$ renews its control value at time $t^{i}_{k}$ and sends $t^{i}_{k}$, $x_i(t^{i}_{k})$ and $u_i(t^{i}_{k})$ to all its out-neighbours immediately.
At each $t_{k}^{i}$, each agent $v_i$ ``pulls'' $x_{j}(t_{k}^{i})$ from agent $v_i$ if $L_{ij}\neq0$.
(This does not mean that agent $v_i$ has to send a request to its in-neighbours at $t^{i}_{k}$ in order to get its in-neighbours' states at $t^{i}_{k}$. Instead, in event-triggered control, agent $v_i$'s in-neighbours has to send its state to agent $v_i$ continuously. And we will also give an algorithm to avoid such continuous communication later.) In order to distinguish it from others, we name this sort of feedback as {\em pull-based}.

Let us recall the model
\begin{align*}
x^{i}(t+1)=f(x^{i}(t))+c_{i}\sum_{j=1}^{m}a_{ij}(f(x^{j}(t)))
\end{align*}
where $\dot{s}(t)=f(s(t))$ is a chaotic oscillator. It was proposed and
investigated in \cite{LC2004} for synchronization of chaotic systems. It
can also be considered as nonlinear consensus model.

As a special case, let $f(x(t))=x(t)$ and $c_{i}=(t_{k+1}^{i}-t_{k}^{i})$,
then
\begin{align*}
x^{i}(t_{k+1}^{i})=x^{i}(t_{k}^{i})+(t_{k+1}^{i}-t_{k}^{i})\sum_{j=1}^{m}a_{ij}x^{j}(t_{k}^{j})
\end{align*}
which is just the event triggering (distributed) model for consensus
problem, though the term "event triggering" was not used. In centralized
control, the bound for $(t_{k+1}^{i}-t_{k}^{i})=(t_{k+1}-t_{k})$ to reach
synchronization was given in that paper when the coupling graph is
indirected (or in \cite{LC2007} for direct
graph), too.

In this paper, the distributed continuous monitoring with pull-based
feedback as the event-triggered controller is considered firstly, namely
agent can observe its in-neighbours' continuous states by its in-neighbours sending their continuous states to it.
It is proved that if the directed network topology is irreducible, then the
pull-based event-triggered coupling strategy can realize consensus for the
multi-agent system. Then we generalize it to the reducible case. By
mathematical induction, it is proved that if the network topology has a
spanning tree, then the pull-based event-triggered coupling strategy can
realise consensus for the multi-agent system, too. Finally the results are
extended to discontinuous monitoring, where each agent computes its next
triggering time in advance without having to receive the system's state
continuously (self-triggered).

In comparison to literature, we have three main contributions: (i) different from \cite{Dvd}-\cite{Mxy} and \cite{Now}, we investigate directed topologies, including irreducible and reducible cases, and we do not make assumption that they are balanced; (ii) different from \cite{Gss}, \cite{Mxy} and \cite{Zzf}, the  event-triggered principles in our paper are fully distributed in the sense that each agent only needs its in-neighbours' state information, especially does not need any a priori knowledge of any global parameter and the Zeno behaviour can be excluded; (iii) different from \cite{Zliu}-\cite{Zzf}, we propose self-triggered principle, by which continuous communication between agents can be avoided.

The paper is organized as follows: in Section \ref{sec2}, some necessary
definitions and lemmas are given; in Section \ref{pull}, the pull-based
event-triggered consensus in multi-agent systems with directed topologies
is discussed; in Section \ref{sec4}, the self-triggered formulation of the
frameworks provided in Section \ref{pull} is presented; in Section
\ref{sec5}, one numerical example is provided to show the effectiveness of
the theoretical results; the paper is concluded in Section \ref{sec6}.

\section{Preliminaries}\label{sec2}
In this section we first review some relating notations, definitions and results on algebraic graph theory \cite{Die,Rah} which will be used later in this paper.

\noindent {\bf Notions}: $\|\cdot\|$ represents the Euclidean norm for
vectors or the induced 2-norm for matrices.  $\bf 1$ denotes the column
vector with each component 1 with proper dimension. $\rho(\cdot)$ stands for
the spectral radius for matrices and $\rho_2(\cdot)$ indicates the minimum
positive eigenvalue for matrices having positive eigenvalues. Given two
symmetric matrices $M,N$, $M>N$ (or $M\ge N$) means $M-N$ is a positive
definite (or positive semi-definite) matrix.

For a weighted directed graph (or digraph) $\mathcal G=(\mathcal V,
\mathcal E, \mathcal A)$ with $m$ agents (vertices or nodes), the set of
agents $\mathcal V =\{v_1,\cdots,v_m\}$, set of links (edges) $\mathcal E
\subseteq \mathcal V \times \mathcal V$, and the weighted adjacency matrix
$\mathcal A =(a_{ij})$ with nonnegative adjacency elements $a_{ij}>0$. A link
of $\mathcal G$ is denoted by $e(i,j)=(v_i, v_j)\in \mathcal E$ if there is
a directed link from agent $v_j$ to agent $v_i$ with weight $a_{ij}>0$, i.e. agent $v_j$ can send
information to agent $v_i$ while the opposite direction transmission might not exist or with different weight $a_{ji}$. The adjacency elements associated with the links of the graph
are positive, i.e., $e(i,j)\in \mathcal E\iff a_{ij}>0$, for all $i,
j\in\mathcal I$. It is assumed that $a_{ii}=0$ for all $i\in \mathcal I$.
Moreover, the in- and out- neighbours set of agent $v_i$ are defined as
\begin{eqnarray*}
N^{in}_i=\{v_j\in \mathcal V\mid a_{ij}>0\},~~~N^{out}_i=\{v_j\in \mathcal V\mid a_{ji}>0\}
\end{eqnarray*}
The in- and out- degree of agent $v_i$ are defined as follows:
\begin{eqnarray*}
deg^{in}(v_i)=\sum\limits_{j=1}^{m}a_{ij},~~~deg^{out}(v_i)=\sum\limits_{j=1}^{m}a_{ji}
\end{eqnarray*}
The degree matrix of digraph $\mathcal G$ is defined as
$D=diag[deg^{in}(v_1), \cdots, deg^{in}(v_m)]$. The weighted Laplacian
matrix associated with the digraph $\mathcal G$ is defined as $L=D-\mathcal
A$. A directed path from agent $v_0$ to agent $v_k$ is a directed graph
with distinct agents $v_0,...,v_k$ and links $e_0,...,e_{k-1}$ such that
$e_i$ is a link directed from $v_i$ to $v_{i+1}$, for all $i<k$.
\begin{definition}
We say a directed graph $\mathcal G$ is strongly connected if for any two
distinct agents $v_{i},~v_{j}$, there exits a directed path from $v_{i}$ to
$v_{j}$.
\end{definition}

By \cite{Rah}, we know that strongly connectivity of $\mathcal G$  is
equivalent to the corresponding Laplacian matrix $L$ is irreducible.
\begin{definition}
We say a directed graph $\mathcal G$ has a spanning tree if there exists at least
one agent $v_{i_{0}}$ such that for any other agent $v_{j}$, there exits a
directed path from $v_{i_{0}}$ to $v_{j}$.
\end{definition}

By Perron-Frobenius theorem \cite{LC2006} (for more detail and proof, see
\cite{LC2007cms}), we have
\begin{lemma}\label{lem2}
If $L$ is irreducible, then $rank(L)=m-1$, zero is an algebraically simple
eigenvalue of $L$ and there is a positive vector
$\xi^{\top}=[\xi_{1},\cdots,\xi_{m}]$ such that $\xi^{\top} L=0$ and
$\sum_{i=1}^{m}\xi_{i}=1$. 
\end{lemma}

Let $\Xi=diag[\xi_{1},\cdots,\xi_{m}]$, 
by the results first given in \cite{LC2006},
we have
\begin{lemma}\label{lem1}
If $L$ is irreducible, then $\Xi L+L^{\top}\Xi$ is a symmetric matrix with
all row sums equal to zeros and has zero eigenvalue with algebraic
dimension one.
\end{lemma}

Here we define some matrices, which will be used later. Let $R=[R_{ij}]_{i,j=1}^{m}$, where
$$R=\frac{1}{2}(\Xi L+L^{\top}\Xi)$$
Obviously, $R$ is positive
semi-definite. Denote the eigenvalue of $R$ by
$0=\lambda_{1}<\lambda_{2}\le\cdots\le\lambda_{m}$, counting the
multiplicities.

We also denote
$$U=\Xi-\xi\xi^{\top}$$

It can also be seen that $U$ has a simple zero eigenvalue and its
eigenvalues (counting the multiplicities) can be arranged as
$0=\mu_{1}<\mu_{2}\le\cdots\le\mu_{m}$. We also denote the eigenvalues of
$L^{T}L$ by
$0=\gamma_{1}<\gamma_{2}\le\cdots\le\gamma_{m}=\rho(L^{\top}L)$. Then, for
all $x\in R^{m}$ satisfying $x\bot 1$, we have
$$\lambda_{2}x^{\top}x\le x^{\top}Rx$$
and
$$x^{\top}UUx\le\mu_{m}^{2}x^{\top}x$$
Therefore, we have
\begin{align}
R\ge\frac{\lambda_{2}}{\mu_{m}^2} UU\label{RU}
\end{align}
\begin{align}
L^{T}L\ge\frac{\gamma_{2}}{\mu_{m}^{2}} UU\label{UULL}
\end{align}
and
\begin{align}
\frac{\lambda_m}{\gamma_{2}}L^{\top}L\ge R\ge
\frac{\lambda_2}{\rho(L^{\top}L)}L^{\top}L\label{RLL}
\end{align}

Pick weight function $\mu(t)>0$ satisfying $\frac{\dot{\mu}(t)}{\mu(t)}\le \beta$.

\section{Pull-based event-triggered principles}\label{pull}
In this section, we consider event-triggered control for multi-agent systems with directed and weighted topology.

{\bf Firstly}, we consider the case of irreducible $L$.

Denote
$q(t)=[q_{1}(t),\cdots,q_{m}(t)]^{\top}$,  where
$q_{i}(t)=-\sum_{j=1}^{m}L_{ij}x_{j}(t)$ and
$f(t)=[f_{1}(t),\cdots,f_{m}(t)]^{\top}$, where
$
f_{i}(t)=q_{i}(t_{k}^{i})-q_{i}(t),~t\in[t_k^{i},t_{k+1}^{i}),~k=1,2,...
$

To depict the trigger event, consider the following candidate Lyapunov function (see \cite{LC2006}):
\begin{align}
V(t)=\frac{1}{2}\sum_{i=1}^{m}\xi_{i}(x_{i}(t)-\bar{x}(t))^{2}
=\frac{1}{2}x^{\top}(t)Ux(t)\label{V}
\end{align}
where $\bar{x}(t)=\sum_{i=1}^{m}\xi_{i}x_{i}(t)$.

By the definition, we have
\begin{align}
\sum_{i=1}^{m}\xi_{i}(x_{i}(t)-\bar{x}(t))
=0
\end{align}
and due to $\xi^{\top} L=0$, we have
\begin{align}
\sum_{i=1}^{m}\xi_{i}L_{ij}x_{j}(t)
=0\end{align}


The derivative of $V(t)$ along (\ref{mg3}) is
\begin{align}
\frac{d}{dt}V(t)=&\sum_{i=1}^{m}\xi_{i}(x_{i}(t)-\bar{x}(t))(\dot{x}_{i}(t)-\dot{\bar{x}}(t))
\nonumber\\
=&\sum_{i=1}^{m}\xi_{i}(x_{i}(t)-\bar{x}(t))\dot{x}_{i}(t)
\nonumber\\=&-
\sum_{i=1}^{m}\xi_{i}(x_{i}(t)-\bar{x}(t))\sum_{j=1}^{m}L_{ij}x_{j}(t_k^{i})
\nonumber\\
=&\sum_{i=1}^{m}\xi_{i}(x_{i}(t)-\bar{x}(t))q_{i}(t_k^{i})\nonumber\\
=&\sum_{i=1}^{m}\xi_{i}(x_{i}(t)-\bar{x}(t))\left\{f_{i}(t)+q_{i}(t)
\right\}\nonumber\\
=&\sum_{i=1}^{m}\xi_{i}(x_{i}(t)-\bar{x}(t))[f_{i}(t)
-\sum_{j=1}^{m}L_{ij}x_{j}(t)]\nonumber\\
=&-\sum_{i=1}^{m}\sum_{j=1}^{m}x_{i}(t)\xi_{i}L_{ij}x_{j}(t)
+\sum_{i=1}^{m}\xi_{i}(x_{i}(t)-\bar{x}(t))f_{i}(t)\nonumber\\
=&-x^{\top}(t)Rx(t)+x^{\top}(t)Uf(t)\nonumber\\
\le&-x^{\top}(t)Rx(t)+\frac{a}{2}x^{\top}(t)UUx(t)+\frac{1}{2a}f^{\top}(t)f(t)\nonumber\\
\le&-(1-\frac{a\mu_m^2}{2\lambda_2})x^{\top}(t)Rx(t)+\frac{1}{2a}f^{\top}(t)f(t)\label{dV3.1}
\end{align}
By (\ref{RLL}), we have
\begin{align}
&\frac{d}{dt}V(t)\nonumber\\
\le&-(1-\frac{a\mu_m^2}{2\lambda_2})\frac{\lambda_2}{\rho(L^{\top}L)}x^{\top}(t)L^{\top}Lx(t)
+\frac{1}{2a}f^{\top}(t)f(t)\nonumber\\
=&-(1-\frac{a\mu_m^2}{2\lambda_2})\frac{\lambda_2}{\rho(L^{\top}L)}q^{\top}(t)q(t)
+\frac{1}{2a}f^{\top}(t)f(t)\nonumber\\
=&\sum_{i=1}^{m}[-(1-\frac{a\mu_m^2}{2\lambda_2})\frac{\lambda_2}{\rho(L^{\top}L)}q_{i}^{2}(t)
+\frac{1}{2a}(q_{i}(t^{i}_{k})-q_{i}(t))^{2}]\label{dV3.11}
\end{align}
and
\begin{align}
&\frac{d[\mu(t)V(t)]}{dt}=\mu(t)\dot{V}(t)+\dot{\mu}(t)V(t)\nonumber\\
\le&\sum_{i=1}^{m}\mu(t)\bigg\{\bigg[-(1-\frac{a\mu_m^2}{2\lambda_2})
\frac{\lambda_2}{\rho(L^{\top}L)}+\frac{\gamma_{2}\dot{\mu}(t)}{\mu_{m}\mu(t)}\bigg]
q_{i}^{2}(t)\nonumber\\
&+\frac{1}{2a}(q_{i}(t^{i}_{k})-q_{i}(t))^{2}\bigg\}
\end{align}
Therefore, we have
\begin{theorem}\label{coro3.1}
Suppose that $\mathcal G$ is strongly connected. $\frac{\dot{\mu}(t)}{\mu(t)}\le \beta(t)$. For $i=1,\cdots,m,$, set $0<a<\frac{2\lambda_2}{\mu_m^2}$, and 
$$b(t)=(1-\frac{a\mu_m^2}{2\lambda_2})\frac{\lambda_2}{\rho(L^{\top}L)}
-\frac{\gamma_{2}\beta(t)}{\mu_{m}}>0$$
\begin{align}
t_{k+1}^{i}=\max_{\tau\ge t_{k}^{i}}\bigg\{&\tau:~\Big|q_{i}(t^{i}_{k})-q_{i}(t)\Big|\nonumber\\
&\le\sqrt{2ab(t)}\Big|q_{i}(t)\Big|,~\forall
t\in[t_{k}^{i},\tau]\bigg\}\label{event3.11}
\end{align}
Then, system (\ref{mg3}) reaches a consensus
\begin{align}
x_{i}(t)-\sum_{j=1}^{m}\xi_{j}x_{j}(t)=
O\bigg(\mu^{-1/2}(t)\bigg)
\end{align}
In addition, for all
$i\in\mathcal I$, we have
and
\begin{align}
\dot{x}_{i}(t)=
O\bigg(\mu^{-1/2}(t)\bigg)
\end{align}
\end{theorem}
\begin{proof} Combining inequalities (\ref{dV3.11}), (\ref{event3.11}) and (\ref{RLL}), we have
\begin{align*}
\frac{d[\mu(t)V(t)]}{dt}\le 0
\end{align*}
for all $t\ge0$. It means
\begin{align*}
V(t)\le \mu(0)\mu^{-1}(t)
\end{align*}
This implies that system (\ref{mg3}) reaches consensus and for all
$i=1,\cdots,m,$
\begin{align*}
x_{i}(t)-\sum_{j=1}^{m}\xi_{j}x_{j}(t)=
O\bigg(\mu^{-1/2}(t)\bigg)
\end{align*}
and
\begin{align*}
\dot{x}_{i}(t)=-\sum_{j=1}^{m}L_{ij}\bigg(x_{j}(t_{k_{i}(t)}^{i})
-\bar{x}(t_{k_{i}(t)}^{i})\bigg)=
O\bigg(\mu^{-1/2}(t)\bigg)
\end{align*}
This completes the proof.
\end{proof}

As special cases, we have
\begin{corollary}\label{coro3.1}
Suppose that $\mathcal G$ is strongly connected. $\mu(t)=e^{\beta t}$. For $i=1,\cdots,m,$, set $0<a<\frac{2\lambda_2}{\mu_m^2}$, and 
$$b=(1-\frac{a\mu_m^2}{2\lambda_2})\frac{\lambda_2}{\rho(L^{\top}L)}
-\frac{\gamma_{2}\beta}{\mu_{m}}>0$$
\begin{align}
t_{k+1}^{i}=\max_{\tau\ge t_{k}^{i}}\bigg\{&\tau:~\Big|q_{i}(t^{i}_{k})-q_{i}(t)\Big|\nonumber\\
&\le\sqrt{2ab}\Big|q_{i}(t)\Big|,~\forall
t\in[t_{k}^{i},\tau]\bigg\}\label{event3.11}
\end{align}
Then, system (\ref{mg3}) reaches a consensus
\begin{align}
x_{i}(t)-\sum_{j=1}^{m}\xi_{j}x_{j}(t)=
O\bigg(e^{-\beta t/2}\bigg)
\end{align}
In addition, for all
$i\in\mathcal I$, we have
\begin{align}
\dot{x}_{i}(t)=
O(e^{-\beta t/2})
\end{align}
\end{corollary}
\begin{corollary}
Suppose that $\mathcal G$ is strongly connected. Set 
\begin{align}
t_{k+1}^{i}=\max_{\tau\ge t_{k}^{i}}\bigg\{&\tau:~\Big|q_{i}(t^{j}_{k})-q_{i}(t)\Big|\nonumber\\
&\le c\Big|q_{i}(t)\Big|,~\forall
t\in[t_{k}^{i},\tau]\bigg\}
\end{align}
or
\begin{align}
t_{k+1}^{i}=\max_{\tau\ge
t_{k}^{i}}\bigg\{&\tau:~\frac{|q_{i}(t^{j}_{k})|}{1+c}
\le\Big|q_{i}(t)\Big|\nonumber\\
&\le \frac{|q_{i}(t^{j}_{k})|}{1-c},~\forall t\in[t_{k}^{i},\tau]\bigg\}
\end{align}
for some sufficient small constant $c$.
Then, system (\ref{mg3}) reaches a consensus.

\end{corollary}

Theorem 1 shows such a constant $c$ does exist.

\begin{remark}
To utilize event-triggering algorithm, two issues should be addressed. Firstly, for any initial condition, at any time $t\ge 0$, under the condition and
the event-triggered principle in Theorem 1, there exists at least one agent
$v_{j_{1}}$, of which the next inter-event time is strictly positive before
consensus is reached.

In fact, suppose that there is no trigger event when $t>T$. Then, we have
\begin{eqnarray}
\dot{x}_{i}(t) =\sum_{j=1}^{m}L_{ij}x_{j}(T_{k_{i}(T)}^{i}),~t>T,~i=\onetom
\end{eqnarray}
which implies
\begin{eqnarray*}
x_{i}(t)-x_{i}(T) =(t-T)\sum_{j=1}^{m}L_{ij}x_{j}(T_{k_{i}(T)}^{i}).
\end{eqnarray*}
By Theorem 1, we have $x_{i_{1}}(t)-x_{i_{2}}(t)\rightarrow 0$. Therefore, for all
$i_{1},i_{2}=\onetom$, we have
\begin{eqnarray*}
\sum_{j=1}^{m}L_{i_{1}j}x_{j}(T_{k_{i_{1}}(T)}^{i_{1}})=
\sum_{j=1}^{m}L_{i_{2}j}x_{j}(T_{k_{i_{2}}(T)}^{i_{2}})
\end{eqnarray*}
and
\begin{eqnarray*}
x_{i_{1}}(T)=x_{i_{2}}(T)
\end{eqnarray*}
which implies $x_{i_{1}}(t)=x_{i_{2}}(t)$ for all $t\ge T$ and $i_{1},i_{2}=\onetom$. It
means that in case there is no triggering time for $t>T$, the consensus has
reached at time $T$.

Secondly, it should be addressed that in any finite interval $[t_{1},t_{2}$, there are only finite triggers. It would be discussed in the following algorithms.
\end{remark}

In the following, we propose another event-triggering setting.

Denote $\delta x_{i}(t)=x_{i}(t)-\bar{x}(t)$, and rewrite
\begin{align}
\frac{d}{dt}V(t)=&\sum_{i=1}^{m}\xi_{i}(x_{i}(t)-\bar{x}(t))[f_{i}(t)
-\sum_{j=1}^{m}L_{ij}x_{j}(t)]\nonumber\\
=&-\sum_{i=1}^{m}\sum_{j=1}^{m}\delta x_{i}(t)\xi_{i}L_{ij}\delta x_{j}(t)
+\sum_{i=1}^{m}\xi_{i}\delta x_{i}(t)f_{i}(t)\nonumber\\
\le&(-\frac{\lambda_{2}}{\max\{\xi_{i}\}}+\frac{a}{2})\sum_{i=1}^{m}\xi_{i}(\delta x_{i}(t))^{2}
+\frac{1}{2a}\sum_{i=1}^{m}\xi_{i}(f_{i}(t))^{2}
\end{align}
and
\begin{align}
&\frac{d[\mu(t)V(t)]}{dt}\nonumber\\
\le&\bigg(-\frac{\lambda_{2}}{\max\{\xi_{i}\}}+\frac{a}{2}
+\frac{\dot{\mu}(t)}{\mu(t)}\bigg)\mu(t)V(t)
+\frac{\mu(t)}{2a}\sum_{i=1}^{m}\xi_{i}(f_{i}(t))^{2}
\end{align}

Firstly, we give a simple lemma.
\begin{lemma}
Suppose a function $V_{1}(t)$ with $V_{1}(0)>0$ and satisfies
$$\dot{V}_{1}(t)\le -c_{1}V_{1}(t)+c_{2}$$
for some constants $c_{1}>0$ and $c_{2}>0$. Then, $V_{1}(t)$ is bounded.
\end{lemma}
In fact, if $V_{1}(t)>c_{2}/c_{1}$, then $\dot{V}_{1}(t)<0$.

\begin{theorem}
Suppose that $\mathcal G$ is strongly connected, function $\mu(t)>0$ satisfies $\dot{\mu}(t)\le \beta \mu(t)$ for some $\beta>0$ and
$$-\frac{\lambda_{2}}{\max\{\xi_{i}\}}+\frac{a}{2}
+\beta<0$$ for some small numbers $a$ and $\beta$. For agent $v_i$, if trigger times $t^{i}_{1}=0,\cdots,t^{i}_{k}$ are known, then use the following trigger strategy to find $t^{i}_{k+1}$:
\begin{align}
t_{k+1}^{i}=\max\Big\{\tau\ge t_{k}^{i}:~&|q_{i}(t^{i}_{k})-q_{i}(t)|^{2}\le \mu^{-1}(t),\forall t\in[t_{k}^{i},\tau]\Big\}\label{event1.1ep}
\end{align}
Then, system (\ref{mg3}) reaches consensus
\begin{align}
|x_{i}(t)-\sum_{j=1}^{m}\xi_{j}x_{j}(t)|\le \frac{\mu^{-1/2}(t)}{\sqrt{2a\xi_{i}(\frac{\lambda_{2}}{\max\{\xi_{i}\}}-\frac{a}{2}
-\beta)}}
\end{align}
In addition, for all
$i\in\mathcal I$, we have
\begin{align}
\dot{x}_{i}(t)=
O(\mu^{-1/2}(t))
\end{align} and the Zeno behaviour could be excluded.

\end{theorem}
\begin{proof} By previous derivations, it is clear that there are two constants $c_{1}>0$ and $c_{2}>0$ such that
\begin{align*}
\frac{d\{\mu(t)V(t)\}}{dt}\le -\bigg\{\frac{\lambda_{2}}{\max\{\xi_{i}\}}-\frac{a}{2}
-\beta\bigg\}\mu(t)V(t)+\frac{1}{2a}
\end{align*}
By Lemma 3, $\mu(t)V(t)$ are bounded. Therefore, for sufficient large $t$, we have
$$\mu(t)V(t)\le \frac{1}{2a(\frac{\lambda_{2}}{\max\{\xi_{i}\}}-\frac{a}{2}
-\beta)}$$
and
$$V(t)\le \frac{\mu^{-1}(t)}{2a(\frac{\lambda_{2}}{\max\{\xi_{i}\}}-\frac{a}{2}
-\beta)}$$
\begin{align}
|x_{i}(t)-\bar{x}(t)|\le \frac{\mu^{-1/2}(t)}{\sqrt{2a\xi_{i}(\frac{\lambda_{2}}{\max\{\xi_{i}\}}-\frac{a}{2}
-\beta)}}
\end{align}
In addition, for all
$i\in\mathcal I$ and $t\in [t_{k_{i}},t_{k_{i}+1}]$, we have
\begin{align*}
|\dot{x}_{i}(t)|=&|\sum_{j=1}^{m}L_{ij}\bigg(x_{j}(t_{k_{i}}^{i})
-\bar{x}(t_{k_{i}}^{i})\bigg)|\\
\le & 2L_{ii}\frac{\mu^{-1/2}(t_{k_{i}})}{\sqrt{2a\xi_{i}
(\frac{\lambda_{2}}{\max\{\xi_{i}\}}-\frac{a}{2}
-\beta)}}
\end{align*}

Furthermore, in any finite length interval $[0,T]$, $||\dot{q}_{i}(t)||^{2}\le M$, and
$\mu(t)$ is bounded. Thus, $M(t^{i}_{k+1}-t^{i}_{k})^{2}\ge|q_{i}(t^{i}_{k+1})-q_{i}(t^{i}_{k})|^{2}=\mu^{-1}(t^{i}_{k+1})$ is lower bounded. Then,
$$(t^{i}_{k+1}-t^{i}_{k})^{2}\ge \frac{\mu^{-1}(t^{i}_{k+1})}{M}\ge \frac{\mu^{-1}(T)}{M}$$
Therefore, in any finite length interval $[0,T]$, there are only finite triggers. It means that Zeno behavior is avoided.
\end{proof}
\begin{theorem}
Suppose that $\mathcal G$ is strongly connected, and
$$-\frac{\lambda_{2}}{\max\{\xi_{i}\}}+\frac{a}{2}
+\beta<0$$ for some small numbers $a$ and $\beta$. For agent $v_i$, if trigger times $t^{i}_{1}=0,\cdots,t^{i}_{k}$ are known, then use the following trigger strategy to find $t^{i}_{k+1}$:
\begin{align}
t_{k+1}^{i}=\max\Big\{\tau\ge t_{k}^{i}:~&|q_{i}(t^{i}_{k})-q_{i}(t)|^{2}\le e^{-\beta t},\forall t\in[t_{k}^{i},\tau]\Big\}\label{event1.1ep}
\end{align}
Then, system (\ref{mg3}) reaches consensus
\begin{align}
|x_{i}(t)-\bar{x}(t)|\le \frac{e^{-\beta t/2}}{\sqrt{2a\xi_{i}(\frac{\lambda_{2}}{\max\{\xi_{i}\}}-\frac{a}{2}
-\beta)}}
\end{align}
and the Zeno behaviour could be excluded; In addition,
$$\dot{x}_{j}(t)=O(e^{-\beta t/2})$$

\end{theorem}
\begin{remark}
(i) In Theorem 2, in order to determine the trigger times, each agent only needs its in-neighbours' state information, especially do not need any a priori knowledge of any global parameter. (ii) By a little more detail analysis, we can show that there exists a constant $c$ such that for each agent $v_i$, $t^i_{k+1}-t^i_{k}\ge c>0$. We omit detail proof here.
\end{remark}
\begin{remark}
By picking different function $\mu(t)$, we can obtain different convergence rate. It can be seen that if $\mu(t)$ increases fast, then the interval $t^i_{k+1}-t^i_{k}$ can be larger, which means less triggers are needed. Instead, if $\mu(t)$ increases slowly, then the interval $t^i_{k+1}-t^i_{k}$ should be smaller, which means more triggers are needed.
\end{remark}
\begin{remark}
The event-triggered principle used  in Theorem 2 may be costly since each agent has to continuously send its state information to its out-neighbours. In the next section, we will give an algorithm to avoid this.
\end{remark}

\section{Distributed self-triggered principles}\label{sec4}

In this section, we extend the pull-based event-triggered principle
discussed in {\bf Section \ref{pull}} to self-triggered case in order to
avoid continuous communication between agents.

Self-triggered approach means that one can predict next triggering time $t_{k}^{i}$ based on the information at previous triggering time $t_{k}^{i}$.

Recall again the model
\begin{align*}
x^{i}(t_{k+1}^{i})=x^{i}(t_{k}^{i})+(t_{k+1}^{i}-t_{k}^{i})\sum_{j=1}^{m}a_{ij}x^{j}(t_{k}^{j})
\end{align*}
In centralized
control, the bound for $(t_{k+1}^{i}-t_{k}^{i})=(t_{k+1}-t_{k})$ to reach
consensus was given in that paper \cite{LC2004} when the coupling graph is
indirected (or in \cite{LC2007} for direct
graph), too. It means that the idea of self-triggering has been considered in these two papers.

For agent $v_{i}$, given $t^{i}_{1}=0,\cdots,t^{i}_{k}$,  its state at $t\in[t^{i}_{k},t^{i}_{k+1}]$ ($t^{i}_{k+1}$ to be determined) can be written as:
\begin{eqnarray}
x_{i}(t)=x_{i}(t^{i}_{k})+(t-t^{i}_{k})q_i(t^{i}_{k}),~t\in[t^{i}_{k},t^{i}_{k+1}]\label{mg13}
\end{eqnarray}
Since each agent $v_p\in N^{in}_i$ sends trigger information to agent $v_{i}$ whenever agent $v_p$ triggers, then at any given time point $r$, agent $v_i$ can predict agent $v_p$'s state at time $t\ge r$ as
\begin{eqnarray}
x_{p}(t)=x_{p}(t^{p}_{k_{p}(r)})+(t-t^{p}_{k_{p}(r)})q_p(t^{p}_{k_{p}(r)})\label{xixp4}
\end{eqnarray}
until agent $v_p$' next triggering after $s$.

Then, from Theorem 3, we have the following result
\begin{theorem}\label{thm4.11}
Suppose that $\mathcal G$ has spanning trees and $L$ is written in the form
of (\ref{PF}). For agent $v_i$, pick $\phi_i>0$, $\alpha_i>0$. If trigger times $t^{i}_{1}=0,\cdots,t^{i}_{k}$ are known, then use the following trigger strategy to find $t^{i}_{k+1}$:
\begin{enumerate}
\item At time $s=t^{i}_{k}$, substituting (\ref{mg13}) and (\ref{xixp4}) into (\ref{event3.21ep}), and solve the following maximizing problem to find out $\tau^{i}_{k+1}$:
    \begin{align}
\tau_{k+1}^{i}=\max\Big\{\tau\ge s:~&|q_{i}(t^{i}_{k})-q_{i}(t)|\le e^{-\beta t},\forall t\in[s,\tau]\Big\}\label{event3.21eps}
\end{align}
\item In case that some in-neighbours of agent $v_{i}$ triggers at time $t_0\in (s,\tau^{i}_{k+1})$, i.e., agent $v_{i}$ received the renewed information form some of its in-neighbours, then updating $s=t_0$ and go to step (1);
\item In case that any of $v_{i}$'s in-neighbours does not trigger during $(s,\tau^{i}_{k+1})$, then $v_{i}$ triggers at time $t^{i}_{k+1}=\tau^{i}_{k+1}$. The agent $v_{i}$ renews its state at $t=t^{i}_{k+1}$ and sends the renewed information, including $t^{i}_{k+1}$, $x_{i}(t^{i}_{k+1})$ and $q_{i}(t^{i}_{k+1})$, to all its out-neighbours immediately.
\end{enumerate}
then, system (\ref{mg3}) reaches consensus exponentially and the Zeno behaviour could be excluded.
\end{theorem}

\begin{remark}
Obviously, Theorem \ref{thm4.11} can be regarded as an algorithm of Theorem 3, by which the continuous communications between different states can be avoided.
\end{remark}

{\bf Secondly}, we consider the case $L$ is reducible. The following
mathematical methods are inspired by \cite{Ctp}. By proper
permutation, we rewrite $L$ as the following Perron-Frobenius form:
\begin{eqnarray}
L=\left[\begin{array}{llll}L^{1,1}&L^{1,2}&\cdots&L^{1,K}\\
0&L^{2,2}&\cdots&L^{2,K}\\
\vdots&\vdots&\ddots&\vdots\\
0&0&\cdots&L^{K,K}
\end{array}\right]\label{PF}
\end{eqnarray}
where $L^{k,k}$ is with dimension $n_{k}$ and associated with the $k$-th
strongly connected component (SCC) of $\mathcal G$, denoted by $SCC_{k}$,
$k=\onetoK$. Accordingly, define
$x=[x{^{1}}^{T},\cdots,x{^{K}}^{T}]^{T}$, where
$x^{k}=[x_{1}^{k},\cdots,x_{n_{k}}^{k}]^{\top}$.

For agent $v_i\in SCC_k$, i.e., $i=M_{k-1}+1,\cdots,M_{k}$, where $M_0=0,~M_{k}=\sum_{i=1}^{k}n_{i}$, denote the combinational state measurement
$q^{k}_{i}(t)
=-\sum_{j=M_{k-1}+1}^{m}L_{i+M_{k-1},j}x_{j}(t)=-\sum_{j=1}^{m}L_{i+M_{k-1},j}x_{j}(t)=q_{i+M_{k-1}}(t)$. And denote
the combinational measurement error by
$f^{k}_{i}(t)=q^{k}_{i}(t_{l}^{i+M_{k-1}})-q^{k}_{i}(t)$ and $
u^{k}_{i}(t)=q^{k}_{i}(t_{l}^{i+M_{k-1}}),~t\in[t_l^{i+M_{k-1}},t_{l+1}^{i+M_{k-1}})$.

If $\mathcal G$ has spanning trees, then each $L^{k,k}$ is irreducible or
has one dimension and for each $k<K$, $L^{k,q}\ne 0$ for at least one
$q>k$. Define an auxiliary matrix
$\tilde{L}^{k,k}=[\tilde{L}^{k,k}_{ij}]_{i,j=1}^{n_{k}}$ as
\begin{eqnarray*}
\tilde{L}^{k,k}_{ij}=\begin{cases}L^{k,k}_{ij}&i\ne j\\
-\sum_{p=1,p\not=i}^{n_{k}}L^{k,k}_{ip}&i=j\end{cases}
\end{eqnarray*}
Then, let
$D^{k}=L^{k,k}-\tilde{L}^{k,k}=diag[D^{k}_{1},\cdots,D^{k}_{n_{k}}]$, which
is a diagonal semi-positive definite matrix and has at least one diagonal
positive (nonzero).

Let ${\xi^{k}}^{\top}$ be the positive left eigenvector of the irreducible
$\tilde{L}^{k,k}$ corresponding to the eigenvalue zero and has the sum of components equaling to $1$.
Denote $\Xi^{k}=diag[\xi^{k}]$. By the structure, it can be seen that
$R^{k}=\frac{1}{2}[\Xi^{k}\tilde{L}^{k,k}+(\Xi^{k}\tilde{L}^{k,k})^{\top}]$
has zero row sums and has zero eigenvalue with algebraic dimension one.
Then, we have
\begin{property}\label{p1}
Under the setup above, $Q^{k}=\frac{1}{2}[\Xi^{k}L^{k,k}+(\Xi^{k}L^{k,k})^{\top}]=R^{k}+\Xi^{k}D^{k}$ is positive definite and $\Xi^{k}\le\frac{\rho(\Xi^{k})}{\rho_2(Q^{k})}Q^{k}$ for all $k<K$.
\end{property}

And let $U^{K}=\Xi^{K}-\xi^{K}(\xi^{K})^{\top}$ and in order to facilitate the presentation, also denote $U^{k}=\Xi^{k},~k=1,\cdots,K-1$.

Now we are going to determine the triggering times for the system
(\ref{mg3}) to reach consensus. Firstly, applying Theorem 1 to the $K$-th
SCC, we can conclude that the $K$-th SCC can reach a consensus with the
agreement value $\nu(t)=\sum_{p=1}^{n_{K}}\xi^{K}_{p}x^{K}_{p}(t)$
and $\lim_{t\to\infty}\dot{\nu}(t)=0$ exponentially.

Then, inductively, consider the $K-1$-th SCC.
We will prove that $\lim_{t\to\infty}|x_{p}^{K-1}(t)-\nu(t)|=0$, for all $p=1,\cdots,n_{K-1}$.

Construct a candidate Lyapunov function as follows
\begin{align}
V_{K-1}(t)=\frac{1}{2}(x^{K-1}(t)-\nu(t){\bf
1})^{\top}\Xi^{K-1}(x^{K-1}(t)-\nu(t){\bf 1})\label{VK-1}
\end{align}

Differentiate $V_{K-1}(t)$ along (\ref{mg3}), we have
\begin{align}
&\frac{d}{dt}V_{K-1}(t)\nonumber\\
=&(x^{K-1}(t)-\nu(t){\bf
1})^{\top}\Xi^{K-1}\Big\{f^{K-1}(t)+q^{K-1}(t)-\dot{\nu}(t){\bf 1}\Big\}\nonumber\\
=&(x^{K-1}(t)-\nu(t){\bf
1})^{\top}\Xi^{K-1}\Big\{f^{K-1}(t)-\dot{\nu}(t){\bf 1}\nonumber\\
&-L^{K-1,K-1}(x^{K-1}(t)-\nu(t){\bf 1})-L^{K-1,K}(x^{K}(t)-\nu(t){\bf 1})\Big\}\nonumber\\
=&(x^{K-1}(t)-\nu(t){\bf
1})^{\top}\Xi^{K-1}f^{K-1}(t)
\nonumber\\
&-[x^{K-1}(t)-\nu(t){\bf 1}]^{\top}\Xi^{K-1}L^{K-1,K-1}[x^{K-1}(t)-\nu(t){\bf 1}]\nonumber\\
&-(x^{K-1}(t)-\nu(t){\bf
1})^{\top}\Xi^{K-1}\left\{L^{K-1,K}(x^{K}(t)-\nu(t){\bf 1})\right\}\nonumber\\
&-(x^{K-1}(t)-\nu(t){\bf
1})^{\top}\Xi^{K-1}\left\{\dot{\nu}(t){\bf 1}\right\}
\nonumber\\
=&W^{K-1}_{0}(t)-W^{K-1}_{1}(t)-W^{K-1}_{2}(t)-W^{K-1}_{3}(t)\label{dV3K-10}
\end{align}
where
\begin{align}
&W^{K-1}_{0}(t)=(x^{K-1}(t)-\nu(t){\bf
1})^{\top}\Xi^{K-1}f^{K-1}(t)\nonumber\\
&\le a_{K-1}V_{K-1}(t)+\frac{1}{2a_{K-1}}\sum_{i=1}^{n_{K-1}}\xi_i^{K-1}[f_i^{K-1}]^2\label{K0}
\end{align}
with any $a_{K-1}>0$,
\begin{align}
&W^{K-1}_{1}(t)\nonumber\\
&=[x^{K-1}(t)-\nu(t){\bf 1}]^{\top}\Xi^{K-1}L^{K-1,K-1}[x^{K-1}(t)-\nu(t){\bf 1}]\nonumber\\
&=[x^{K-1}(t)-\nu(t){\bf 1}]^{\top}Q^{K-1,K-1}[x^{K-1}(t)-\nu(t){\bf 1}]\label{K1}\\
&W^{K-1}_{2}(t)=(x^{K-1}(t)-\nu(t){\bf 1})^{\top}\Xi^{K-1}L^{K-1,K}(x^{K}(t)-\nu(t){\bf 1})\nonumber\\
&W^{K-1}_{3}(t)=(x^{K-1}(t)-\nu(t){\bf 1})^{\top}\Xi^{K-1}(\dot{\nu}(t){\bf 1})\nonumber
\end{align}

By Cauchy inequality, for any
$\upsilon^{K-1}_2>0,~\upsilon^{K-1}_3>0$, we have
\begin{align}
-W^{K-1}_{2}(t)\le \upsilon^{K-1}_2V_{K-1}(t)+F^{K-1}_{2}(t)\nonumber\\
-W^{K-1}_{3}(t)\le \upsilon^{K-1}_3V_{K-1}(t)+F^{K-1}_{3}(t)\label{W1}
\end{align}
where
\begin{align*}
F^{K-1}_{2}(t)&=\frac{1}{4\upsilon^{K-1}_2}\sum_{i=1}^{n_{K-1}}\xi_{i}^{K-1}
\left\{\sum_{p=1}^{n_{K}}L^{K-1,K}_{i,p}
[x^{K}_{p}(t)-\nu(t)]\right\}^{2}\\
F^{K-1}_{3}(t)&=\frac{1}{4\upsilon^{K-1}_2}\sum_{i=1}^{n_{K-1}}\xi_{i}^{K-1}[\dot{\nu}(t)]^2
=\frac{1}{2\upsilon^{K-1}_3}[\dot{\nu}(t)]^2
\end{align*}

According to the discussion of $SCC_{K}$ and Theorem \ref{coro1.1ep}, for all
$p=1,\cdots,n_{K}$, we have
$
\lim_{t\to\infty}x^{K}_{p}(t)-\nu(t)=0,~\lim_{t\to\infty}\dot{\nu}(t)=0
$
exponentially. Thus
\begin{align}\label{F2up}
\lim_{t\to\infty}F^{K-1}_{2}(t)=0,~\lim_{t\to\infty}F^{K-1}_{3}(t)=0
\end{align}
exponentially.

Thus, (\ref{dV3K-10}) can be rewritten as follows
\begin{align}
&\frac{d}{dt}V_{K-1}(t)
\le  a_{K-1}V_{K-1}(t)+\frac{1}{2a_{K-1}}\sum_{i=1}^{n_{K-1}}\xi_i^{K-1}[f_i^{K-1}]^2\nonumber\\
&-W^{K-1}_{1}(t)-W^{K-1}_{2}(t)-W^{K-1}_{3}(t)\nonumber\\
\le&-\bigg[1-\frac{a_{K-1}\rho(\Xi^{K-1})}{2\rho_2(Q^{K-1})}\bigg]W^{K-1}_{1}(t)-W^{K-1}_{2}(t)\nonumber\\
&+\frac{1}{2a_{K-1}}\sum_{i=1}^{n_{K-1}}\xi_i^{K-1}[f_i^{K-1}]^2-W^{K-1}_{3}(t)\label{dV3K-2}
\end{align}
Thus, we have
\begin{theorem}\label{coro3.2ep}
Suppose that $\mathcal G$ has spanning trees and $L$ is written in the form
of (\ref{PF}). For agent $v_i$, if trigger times $t^{i}_{1}=0,\cdots,t^{i}_{k}$ are known, then use the following trigger strategy to find $t^{i}_{k+1}$:
\begin{align}
t_{k+1}^{i}=\max\Big\{\tau\ge t_{k}^{i}:~&|q_{i}(t^{i}_{k})-q_{i}(t)|\le e^{-\beta t},\forall t\in[t_{k}^{i},\tau]\Big\}\label{event3.21ep}
\end{align}
system (\ref{mg3}) reaches consensus exponentially and the Zeno behaviour could be excluded.
\end{theorem}
\begin{proof}
If $v_i\in K$-th SCC, the event-triggered rule (\ref{event3.21ep}) is the same
as (\ref{event1.1ep}) in Theorem 3, since $L$ is written in the
form of (\ref{PF}). By Theorem 3, we can conclude that under
the updating rule of $\{t^{j+M_{K-1}}_{l}\}$for all $j=1,\cdots,n_K$ and $\lim_{t\to\infty}\dot{\nu}(t)=0$, the subsystem restricted in $SCC_{K}$ reaches a consensus.
Additionally,
$\lim_{t\to\infty}|x^{K}_{i}(t)-\nu(t)|=0$
for all $i=1,\cdots,n_K$ and $\lim_{t\to\infty}\dot{\nu}(t)=0$ as well.

In the following, we are to prove that the state of the agent
$v_{p+M_{K-2}}\in SCC_{K-1}$ converges to $\nu(t)$. The remaining can be
proved similarly by induction.

From (\ref{dV3K-2}) and the inequality (\ref{event3.21ep}), we have
\begin{align*}
\frac{d}{dt}V_{K-1}(t)
\le&-[1-\frac{a_{K-1}\rho(\Xi^{K-1})}{2\rho_2(Q^{K-1})}\bigg]\frac{\rho_2(Q^{K-1})}{\rho(U^{K-1})}V_{K-1}(t)\nonumber\\
&+(\upsilon^{K-1}_2+\upsilon^{K-1}_3)V_{K-1}(t)+W^{K-1}_{4}(t)
\end{align*}
where
\begin{align*}
W^{K-1}_{4}(t)=F^{K-1}_{2}(t)+F^{K-1}_{3}(t)+\frac{1}{2a_{K-1}}\sum_{j=1}^{n_{K-1}}\xi_j^{K-1}\delta_j^2(t)
\end{align*}
Picking $a_{K-1}=\frac{\rho_2(Q^{K-1})}{\rho(\Xi^{K-1})}$ and sufficiently small $\upsilon^{K-1}_2$ and $\upsilon^{K-1}_3$, there exists some $\varepsilon_{K-1}>0$ such that
\begin{align*}
\frac{d}{dt}V_{K-1}(t)
\le-\varepsilon_{K-1} V_{K-1}(t)+W^{K-1}_{4}(t)
\end{align*}
Thus
\begin{align*}
V_{K-1}(t)
\le e^{-\varepsilon_{K-1}t}\bigg\{V_{K-1}(0)+\int_{0}^{t}e^{\varepsilon_{K-1}s}W^{K-1}_{4}(s)ds\bigg\}
\end{align*}
From (\ref{F2up}), we have $\lim_{t\to\infty}W^{K-1}_{4}(t)=0$ exponentially. Thus, we have $\lim_{t\to\infty}V_{K-1}(t)=0$ exponentially.This implies that system (\ref{mg3}) reaches a consensus and $\lim_{t\to\infty}|x_{p}^{K-1}(t)-\nu(t)|=0$ exponentially for all $p=1,\cdots,n_{K-1}$.

Similar to the proof in Theorem 3, we can prove that the Zeno behaviour can be excluded for agent $v_i\in K-1$-th SCC.

Then, we can complete the proof by induction to $SCC_{k}$ for $k<K-1$.
\end{proof}

\section{Examples}\label{sec5}
In this section, one numerical example is given to demonstrate the effectiveness of the presented results.

Consider a network of seven agents with a directed reducible Laplacian matrix
\begin{eqnarray*}
L=\left[\begin{array}{rrrrrrr}-12&0&5&2&5&0&0\\
3&-8&3&0&0&0&2\\
0&4&-12&3&0&5&0\\
0&0&6&-11&1&4&0\\
0&0&0&0&-7&2&5\\
0&0&0&0&5&-6&1\\
0&0&0&0&0&8&-8
\end{array}\right]
\end{eqnarray*}
with a spanning tree described by Figure \ref{fig:1}. The seven agents can be divided into two strongly connected components, i.e. the first four agents form a strongly connected component and the rest form anther. The initial value of each agent is also randomly selected within the interval $[-5,5]$ in our simulation. Figure \ref{fig:3} (a) shows the evolution of the Lyapunov function $V(t)=V_1(t)+V_2(t)$ (see (\ref{VK-1})), and Figure \ref{fig:3} (b) illustrates the trigger times of each agents under the self-triggered principles provided in Theorem \ref{thm4.11} with $\phi_i=20$ and $\alpha_i=1.5,~i=1,\cdots,7$, and initial value $[2.192,-3.699,-2.982,4.726,3.575,4.074,-3.424]^{\top}$. It can be seen that under the self-triggering principle in Theorem \ref{thm4.11}, $V(t)$ approaches 0 exponentially and the inter-event times of each agent are strictly bigger than some positive constants.

\begin{figure}[hbt]
\centering
\includegraphics[width=0.4\textwidth]{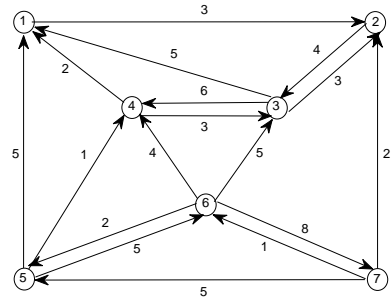}
\caption{The communication graph.}
\label{fig:1}
\end{figure}

\begin{figure}[hbt]
\centering
\includegraphics[width=0.51\textwidth]{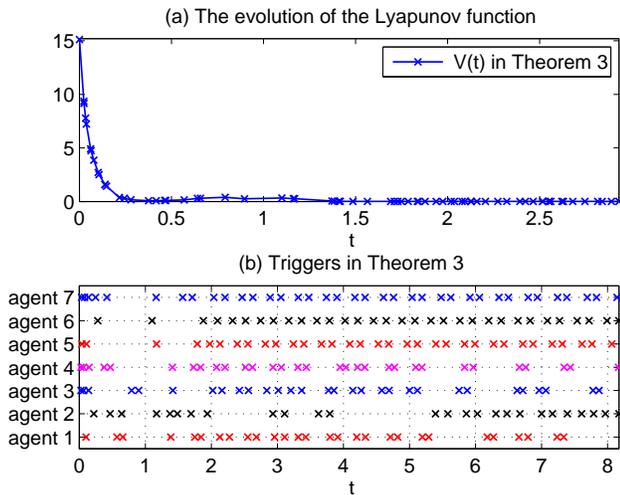}
\caption{The evolution of the Lyapunov function and the trigger times of each agents..}
\label{fig:3}
\end{figure}

\section{Conclusions}\label{sec6}
In this paper, we present distributed event-triggered and self-triggered
principles in for multi-agent systems with general directed topologies. We
derive pull-based event-triggered principles: in case the graph is reducible with a spanning tree, the triggering
time of each agent given by the inequality (\ref{event3.21ep}) only depends on
the states of each agent's in-neighbors. It is shown that with those
principles, consensus can be reached exponentially, and Zeno behavior can
be excluded. The results then are extended to discontinuous monitoring,
where each agent computes its next triggering time in advance without
having to observe the system��s state continuously. The effectiveness of the
theoretical results are verified by one numerical example.

\end{document}